\documentclass[11pt,reqno]{amsart}
\usepackage{graphicx}
\usepackage{float}
\usepackage[caption = false]{subfig}
\usepackage{enumerate}
\usepackage{mathtools}
\usepackage{breqn}
\usepackage{array, geometry, graphicx}
\usepackage{amsmath,amsfonts,amssymb,amsthm,mathrsfs,wasysym}
\usepackage{pdfpages}
\newtheorem{theorem}{Theorem}[section]
\newtheorem{corollary}[theorem]{Corollary}
\newtheorem{lemma}[theorem]{Lemma}

\theoremstyle{definition}
\newtheorem{definition}[theorem]{Definition}
\theoremstyle{remark}

\numberwithin{equation}{section}
\DeclareMathOperator{\RE}{Re}

\begin{document}

\title{A new differential subordination technique for a subclass of starlike functions}
\author{S. Sivaprasad Kumar}
\address{Department of Applied Mathematics, Delhi Technological University, Delhi--110042, India}
\email{spkumar@dce.ac.in}
\author[Pooja Yadav]{Pooja Yadav}
\address{Department of Applied Mathematics, Delhi Technological University, Delhi--110042, India}
\email{poojayv100@gmail.com}

\subjclass[2010]{30C45, 30C80}

\keywords{Analytic functions,  Starlike functions, Subordination, Admissibility condition}
\begin{abstract}
In the present investigation, we employ a new technique to find several first and second order differential subordination implications involving the following starlike class associated with a bean shaped domain: \begin{equation*}
\mathcal{S}^*_{\mathfrak{B}}:=\left\{f\in\mathcal{S}:\dfrac{zf'(z)}{f(z)}\prec\sqrt{1+\tanh{z}}=:\mathfrak{B}(z)\right\}.
	\end{equation*} Also, we give several applications stemming from our derived results.
\end{abstract}

\maketitle

\section{Introduction}
Let $\mathbb{D}=\{z\in\mathbb{C}:|z|<1\}$ and $\mathcal{H}$ be the class of  functions $f$ analytic in $\mathbb{D}$ with $f(z)=1+a_{1}z+a_{2}z^{2}+\cdots.$ Also, let  $\mathcal{A}$ be the class of  analytic functions $f$ normalized by $f(0)=f'(0)-1=0$ in $\mathbb{D}$ and $\mathcal{S}$ be the subclass of $\mathcal{A}$
	consisting of all univalent functions. Further $\mathcal{S}^*$ is a subclass of $\mathcal{S}$, which consists of all functions $f$ such that $f(\mathbb{D})$ is starshaped with respect to origin and these functions are known as starlike functions with respect to origin.   Analytically, $f\in\mathcal{S}^{*}$  if we have $\RE zf'(z)/f(z)>0$
	Furthermore, to compare the inclusion properties of two analytic functions $f$ and $g$, having same origin, i.e. $f(0) = g(0)$, we say that $f$ is subordinate to 
	$g$, written as $f\prec g$, if and only if $f(z) = g(w(z))$ where $w(z)$ is a Schwarz function.
	 In terms of subordination, Ma and Minda \cite{maminda} unified all the subclasses of $\mathcal{S}^*$ by defining the following class:
	\begin{equation}\label{sp}
	\mathcal{S}^{*}(\phi)=\bigg\{f\in\mathcal{S}:\frac{zf'(z)}{f(z)}\prec\phi(z)\bigg\},
	\end{equation} 
	where $\phi$ is a univalent and analytic function with positive real part, starlike with respect to  $\phi(0) = 1$,  $\phi'(0) > 0$ and  $\phi(\mathbb{D})$ is symmetric with respect to the real axis.
Several well-known classes can be obtained  by specializing  $\phi$, such as the Janowski strongly starlike class of order $\alpha$ $\mathcal{SS}^*(A,B,\alpha):=\mathcal{S}^*(((1+Az)/(1+Bz))^\alpha)$ where $\alpha\in(0,1]$ and $A,B\in\mathbb{C}$ with $|B|\leq1$ and $A\neq B$, studied by Kumar and Yadav in \cite{siva}.   Here $\mathcal{SS}^*(1,-1,\alpha)=:\mathcal{SS}^*(\alpha)$; the class of strongly starlike functions of order $\alpha$. And for $A,B\in[-1,1]$ and $\alpha=1$ we obtain the class of Janowski starlike functions, denoted as $\mathcal{S}^*(A,B)$, which specifically
reduces to many well-known classes, for instance see \cite{janowski}. Sok\'{o}l and Stankiewicz \cite{sokol1} considered the class $\mathcal{S}^*_L$ by opting $\phi(z)=\sqrt{1+z}.$ Geometrically $\mathcal{S}^*_L=\{f\in\mathcal{S};|(zf'(z)/f(z))^2-1|<1\}$, that is, $zf'(z)/f(z)$ lies in the domain bounded by the right lemniscate of Bernoulli.  Moreover,  Goel and Kumar \cite{pri} studied the starlike class associated with sigmoid, $\mathcal{S}^*(2/(1 + e^{-z}))=:\mathcal{S}^*_{SG}$ and Mendiratta et al. \cite{exp} introduced and studied the class associated with exponential, $\mathcal{S}^*(e^{z})=:\mathcal{S}^*_e$. Similarly, in \cite{bean} the following class was introduced by Kumar and Yadav: 
\begin{equation*}
\mathcal{S}^*_{\mathfrak{B}}:=\left\{f\in\mathcal{S}:\dfrac{zf'(z)}{f(z)}\prec\sqrt{1+\tanh{z}}=:\mathfrak{B}(z)\right\}.
	\end{equation*} 
Observe that the function $\mathfrak{B}(z):=\sqrt{1+\tanh{z}}=\sqrt{2/(1+e^{-2z})}$ conformally maps $\mathbb{D}$ onto the following bean shaped region \begin{equation}\label{domain}\Omega_{\mathfrak{B}}:=\left\{w\in\mathbb{C}:\left|\log\left(\dfrac{w^2}{2-w^2}\right)\right|<2\right\}.\end{equation}
In \cite{bean}, authors obtained various geometric characteristics of $\mathfrak{B}(z)$ and derived several inclusion and radii results associated with $\mathcal{S}^*_{\mathfrak{B}}.$
 The present article is an extension of that work in terms of differential subordination. Here,  we are establishing certain first and second order differential subordination implications related to the function $\mathfrak{B}(z)$, utilizing a lemma that significantly simplifies the calculations. Differential subordination, analogous to differential inequalities in real analysis, has numerous applications. A pivotal contribution to this field was made by Miller and Mocanu's influential monograph \cite{miller}, which sparked a major revolution in the geometric function theory research community.
Many researchers have addressed problems related to differential subordination, such as those in \cite{cho,priyanka,miller2,salu,neha}. Chapter 2 of the aforementioned monograph presents a set of results that establish admissibility conditions for second-order differential subordination. Significant work has been built upon these results. For instance, Madaan et al. \cite{madaan} investigated the class of admissible functions associated with the lemniscate of Bernoulli and proved several first- and second-order differential subordination implications. In \cite{neha}, the author established various second- and third-order differential subordination implications by defining admissibility conditions for the exponential function. In \cite{ali}, the authors modified the concept of second-order differential subordination by introducing $\beta$-admissible functions, and in \cite{antonino}, the authors extended the concept of differential subordination and admissibility conditions to third order.

We now present essential definitions, lemmas, and notations necessary for further discussion.
\begin{lemma}\cite{bean}\label{subradi}
	 Let  $w\in\mathbb{C}$ such that $|w|>\sqrt{2}$ and $R_{0}=e\sqrt{2/(e^2-1)}\simeq1.5208$. Then
	\begin{equation*}
	\bigg|\log\bigg(\frac{w^2}{2-w^2}\bigg)\bigg|\geq2,\quad\text{whenever} \quad |w|\geq R_{0}.
	\end{equation*}
\end{lemma}

\begin{definition}\cite{miller}
     Let $\mathcal{Q}$ denotes the set of functions $q$ that are analytic and univalent on $\overline{\mathbb{D}}\backslash \mathcal{E}(q),$ where $$\mathcal{E}(q)=\left\{\zeta\in\partial\mathbb{D}:\lim_{z\to\zeta}q(z)=\infty\right\}$$
  such that $q'(\zeta)\neq0$ for $\zeta\in\partial\mathbb{D}\backslash \mathcal{E}(q).$ \end{definition}
  \begin{definition}\cite{miller} Let $\Omega$ be a set in $\mathbb{C}$, $q\in \mathcal{Q}$ and $n$ be a positive integer, the class of admissible functions $\Psi_{n}(\Omega,q),$ consists of those functions $\psi:\mathbb{C}^3\times\mathbb{D}\to\mathbb{C}$ that satisfy the admissibility condition $\psi(r,s,t;z)\not\in\Omega$, whenever  $r=q(\zeta),$ $s=m\zeta q'(\zeta),$
 $$\RE\left(\dfrac{t}{s}+1\right)\geq m\RE\left(\dfrac{\zeta q''(\zeta)}{q'(\zeta)}+1\right),$$
 $z\in\mathbb{D},$ $\zeta\in\partial\mathbb{D}\backslash \mathcal{E}(q)$ and $m\geq n$.  In particular, we denote $\Psi_1(\Omega,q)$ as $\Psi(\Omega,q).$ 
 \end{definition}
 \begin{lemma}\cite{miller}
Let $\psi\in\Psi(\Omega,q)$ with $q(0)=1$. If $p\in\mathcal{H}$ satisfies 
$$\psi(p(z),zp'(z),z^2p''(z);z)\in\Omega,$$ then $p\prec q.$
 \end{lemma}
 Throughout this article, we consider the function $q(z)=\mathfrak{B}(z)$ and define the admissibility class $\Psi(\Omega,q),$ where $\Omega\subset\mathbb{C}$.
 Since $\mathcal{E}(\mathfrak{B})=\varnothing$, therefore $\mathfrak{B}\in\mathcal{Q}$. For $\zeta=e^{i \theta}\;(0\leq\theta<2\pi)$ we have, 
 $$q(\zeta)=\sqrt{\dfrac{2}{1+e^{-2e^{i\theta}}}},\;\;\;\zeta q'(\zeta)=\dfrac{\sqrt{2}e^{i\theta}e^{-2e^{i\theta}}}{(1+e^{-2e^{i\theta}})^{3/2}}$$ and
 \begin{equation}\label{gt}
     \RE\left(\dfrac{\zeta q''(\zeta)}{q'(\zeta)}\right)=\dfrac{\cos{\theta}-2e^{4\cos{\theta}}\cos{\theta}+e^{2\cos{\theta}}(\cos(\theta-2\sin{\theta})-2\cos(\theta+2\sin{\theta}))}{1+e^{4\cos{\theta}}+2e^{2\cos{\theta}}\cos(2\sin{\theta})}=:g(\theta).
 \end{equation}
 Here, \begin{equation}\label{ot}|q(\zeta)|=\dfrac{\sqrt{2}}{(1+e^{-4\cos\theta}+2e^{-2\cos\theta}\cos(2\sin\theta))^{\frac{1}{4}}}=:\omega(\theta).\end{equation}
From \cite{bean}, we obtain that the minimum value of $\omega(\theta)$ is  $\omega(\pi)=\sqrt{2/(1+e^2)}$ and maximum is $\omega(\theta_0)=:\omega_{\theta_0}\simeq1.438$, where $\theta_0\simeq1.31364$.
 By simple analysis we also obtain that  $\min g(\theta)=g(0)$, and
 \begin{equation}\label{dt}
    |\zeta q'(\zeta)|= \dfrac{\sqrt{2}e^{-2\cos{\theta}}}{(1+e^{-4\cos{\theta}}+2e^{-2\cos{\theta}}\cos(2\sin{\theta}))^{3/4}}:=d(\theta).
 \end{equation}
  Further, a computation shows that  $d'(\theta)=0$ if and only if  $\theta\in\{0,\theta_1,\pi\}$ where $\theta_1\simeq1.639$.  We observe that $d(\theta)$ is increasing in $[0,\theta_1]$ and  decreasing in $[\theta_1,\pi]$.
  Therefore, $\max d(\theta)=d(1.639)=:d_{\theta_1}\simeq0.905$ and  $\min d(\theta)=\min\{d(0),d(\pi)\}=\min\{0.158,0.304\}=d(0)$.
 Moreover,  $\Psi(\Omega,\mathfrak{B})$ is defined to be the class of functions $\psi:\mathbb{C}^3\times \mathbb{D}\to\mathbb{C}$, which satisfies the following conditions:
 $$\psi(r,s,t;z)\not\in\Omega,$$
 whenever
 $$r=\mathfrak{B}(\zeta)=\sqrt{\dfrac{2}{1+e^{-2e^{i\theta}}}};\;\;\;s=m\zeta\mathfrak{B}'(\zeta)=\dfrac{\sqrt{2}me^{i\theta}e^{-2e^{i\theta}}}{(1+e^{-2e^{i\theta}})^{3/4}};\;\;\;\RE\left(1+\dfrac{s}{t}\right)=m(1+g(\theta)),$$
 where $z\in\mathbb{D}$, $0\leq\theta<2\pi$ and $m\geq1.$
 Here, we start from first order differential subordination.
	
\section{First order differential subordination}


We first establish the following lemmas to prove our main results:
\begin{lemma}\label{janlemma}
Let $-1< B\leq0<A\leq 1$ and $w\in\mathbb{C}$ such that $|w|>1$, then 
$$\left|\dfrac{w-1}{A-Bw}\right|\geq1, \;\;\;\text{whenever}\;\;\;  
|w|\geq\dfrac{1+A}{1+B}.$$
\end{lemma}
\begin{proof}
	Let $w=R e^{i \theta}$ where $R>1$ and $\theta\in(0,2\pi].$ Then
		\begin{gather*}
	\left|\dfrac{w-1}{A-Bw}\right|\geq1
	\intertext{is equivalent to}
	g(R,\theta):=\dfrac{1+R^2-2R\cos{\theta}}{A^2+B^2R^2-2ABR\cos{\theta}}\geq1.
	\end{gather*}A simple computation shows that $g(R,\theta)$ attains its minimum at $\theta=0$. Thus the above inequality holds if $g(R,0)\geq1$,
i.e. \begin{equation}\label{lem}
\dfrac{(R-1)^2}{(A-BR)^2}\geq1.
\end{equation}Since $|w|=R>1$ and $B\leq0$  then $BR\leq B<A,$ thus (\ref{lem}) reduces to 
$$\dfrac{R-1}{A-BR}\geq1\implies R\geq\dfrac{1+A}{1+B}.$$
Thus the result holds.
\end{proof}
\begin{lemma}\label{lemlemma}
Let  $w\in\mathbb{C}$ such that $|w|>1$, then 
$$\left|w^2-1\right|\geq1, \;\;\;\text{whenever}\;\;\;  
|w|\geq\sqrt{2}.$$
\end{lemma}
\begin{proof}
	Let $w=R e^{i \theta}$ where $R>1$ and $\theta\in(0,2\pi].$ Then
		\begin{gather*}
	\left|w^2-1\right|\geq1
	\intertext{is equivalent to}
	L(R,\theta):=1+R^4-2R^2\cos{2\theta}\geq1.
	\end{gather*}
 Clearly, the above inequality holds if $\min L(R,\theta)\geq1$. Since, $\min L(R,\theta)= L(R,0)$  and $R>1$, thus 
 the result holds.
\end{proof}
 In this section, we find conditions on $\psi(p(z),zp'(z);z)$, so that the following implication holds:
 \begin{equation*}
     \psi(p(z),zp'(z);z)\prec h(z)\implies p(z)\prec\mathfrak{B}(z),
 \end{equation*}where $h(z)$ is either $\mathfrak{B}(z)$, $(1+Az)/(1+Bz)$ $(-1< B\leq0< A\leq1)$ or $\sqrt{1+z}$. Here, $\psi(p(z),zp'(z);z)$ is taken to be as follows: 
 $(p(z))^\delta+\beta(zp'(z))^n$, $(p(z))^\delta+\beta zp'(z)/(p(z))^n$ and $p(z)+zp'(z)/(\beta p(z)+\gamma)$.

\begin{theorem}\label{1stthm1}
Let $n$ be a positive integer, $\delta\in\{0,1\}$, $R_{0}=e\sqrt{2/(e^2-1)}$, $\omega_{\theta_0}\simeq1.438$ and  $\beta\in\mathbb{C}$ such that $|\beta| \geq (R_0+\omega_{\theta_0}^\delta)(1+e^2)^{3n/2}/(\sqrt{2}e)^n$. Suppose $p(z)\in\mathcal{H}$
 and satisfies the subordination
\begin{equation*}
    (p(z))^\delta+\beta (zp'(z))^n\prec \mathfrak{B}(z),
\end{equation*}
then $ p(z)\prec \mathfrak{B}(z).$
\end{theorem}

\begin{proof}
   Let $q(z)=\sqrt{2/(1+e^{-2z})}=:\mathfrak{B}(z)$ and  $\Omega=q(\mathbb{D})=\Omega_{\mathfrak{B}}.$ Now consider $\psi:\mathbb{C}^2\times\mathbb{D}\to\mathbb{C}$ defined by $\psi(r,s;z)=r^\delta+\beta s^n.$ For $\psi$ to lie in the admissible class $\Psi(\Omega_{\mathfrak{B}},\mathfrak{B})$,it must satisfy 
   $$\psi\left(\sqrt{\dfrac{2}{1+e^{-2e^{i \theta}})}},\dfrac{\sqrt{2}m e^{i \theta}e^{-2e^{i \theta}}}{(1+e^{-2e^{i \theta}})^{3/2}};z\right)\not\in\Omega_{\mathfrak{B}}\qquad(z\in\mathbb{D},0\leq\theta<2\pi,m\geq1),$$ 
   which is equivalent to
   $$\left|\log\left(\dfrac{(r^\delta+\beta s^n)^2}{2-(r^\delta+\beta s^n)^2}\right)\right|\geq2.$$ 
    We now use the Lemma \ref{subradi} to prove the result, for which,  let us assume $w=r^\delta+\beta s^n$.
From (\ref{ot}) and (\ref{dt}) respectively, we have $|r|=|\mathfrak{B}(\zeta)|=|q(\zeta)|=\omega(\theta)\leq\max\omega(\theta)=\omega_{\theta_0}$ and $|s|/m=|\zeta\mathfrak{B}'(\zeta)|=|\zeta q'(\zeta)|=d(\theta)\geq\min d(\theta)=d(0)$.  Thus
\begin{align*}
  |w|&=  |r^\delta+\beta s^n|\\&\geq |\beta| |s|^n-|r|^\delta\\
    &\geq |\beta|m^n\left(d(0)\right)^n-\omega_{\theta_0}^\delta.
\end{align*}
Since $m\geq1$ and $|\beta| \geq (R_0+\omega_{\theta_0}^\delta)(1+e^2)^{3n/2}/(\sqrt{2}e)^n$, we have \begin{equation*}
    |w|\geq|\beta|\left(\dfrac{\sqrt{2}e}{(1+e^2)^{3/2}}\right)^n-\omega_{\theta_0}^\delta\geq R_0>\sqrt{2},
\end{equation*} therefore the result follows at once by Lemma \ref{subradi}.
\end{proof}

\begin{corollary}
    Let $n$ be a positive integer, $\delta\in\{0,1\}$, $R_{0}=e\sqrt{2/(e^2-1)}$, $\omega_{\theta_0}\simeq1.438$ and  $\beta\in\mathbb{C}$ such that $|\beta| \geq (R_0+\omega_{\theta_0}^\delta)(1+e^2)^{3n/2}/(\sqrt{2}e)^n$. Suppose $f\in\mathcal{A}$ such that
    \begin{equation*}
         \left(\dfrac{zf'(z)}{f(z)}\right)^\delta+\beta\left(\dfrac{zf'(z)}{f(z)}\left(1+\dfrac{zf''(z)}{f'(z)}-\dfrac{zf'(z)}{f(z)}\right)\right)^n\prec  \mathfrak{B}(z),
    \end{equation*}then $f\in\mathcal{S}^*_{\mathfrak{B}}.$
\end{corollary}
\begin{proof}
    Let $p(z)=zf'(z)/f(z)$, then 
    $$(p(z))^\delta+\beta (zp'(z))^n=\left(\dfrac{zf'(z)}{f(z)}\right)^\delta+\beta\left(\dfrac{zf'(z)}{f(z)}\left(1+\dfrac{zf''(z)}{f'(z)}-\dfrac{zf'(z)}{f(z)}\right)\right)^n.$$
    Therefore the result holds using Theorem \ref{1stthm1}.
\end{proof}
\begin{theorem}\label{1stthm2}
Let $n$ be a positive integer, $\delta\in\{0,1\}$, $R_{0}=e\sqrt{2/(e^2-1)}$, $\omega_{\theta_0}\simeq1.438$ and  $\beta\in\mathbb{C}$ such that  $|\beta| \geq (R_0+\omega_{\theta_{0}}^\delta)(1+e^2)^{(3-n)/2}/(\sqrt{2}e)^{1-n}$. Suppose $p(z)\in\mathcal{H}$ and satisfies the subordination
\begin{equation*}
    (p(z))^\delta+\beta \dfrac{zp'(z)}{(p(z))^n}\prec  \mathfrak{B}(z),
\end{equation*}then $p(z)\prec \mathfrak{B}(z)$.
\end{theorem}

\begin{proof}
   Let $q(z)=\sqrt{2/(1+e^{-2z})}=:\mathfrak{B}(z)$ and $\Omega=q(\mathbb{D})=\Omega_{\mathfrak{B}}.$ Now consider $\psi:\mathbb{C}^2\times\mathbb{D}\to\mathbb{C}$ be as $\psi(r,s;z)=r^\delta+\beta s/r^n.$ For $\psi$ to lie in the admissible class $\Psi(\Omega_{\mathfrak{B}},\mathfrak{B}),$ we must have   $$\psi\left(\sqrt{\dfrac{2}{1+e^{-2e^{i \theta}})}},\dfrac{\sqrt{2}m e^{i \theta}e^{-2e^{i \theta}}}{(1+e^{-2e^{i \theta}})^{3/2}};z\right)\not\in\Omega_{\mathfrak{B}}\qquad(z\in\mathbb{D},0\leq\theta<2\pi,m\geq1),$$ which is equivalent to  $$\left|\log\left(\dfrac{(r^\delta+\beta s/r^n)^2}{2-(r^\delta+\beta s/r^n)^2}\right)\right|\geq2.$$ 
    We now use the Lemma \ref{subradi} to prove the result, for which,  let us assume $w=r^\delta+\beta s/r^n$.
From (\ref{ot}), we have   $|r|=\omega(\theta)\leq\omega_{\theta_0}$. Thus
\begin{align*}
  |w|&= \bigg |r^\delta+\beta \dfrac{s}{r^n}\bigg|\\&\geq |\beta| \left(\dfrac{2^{\frac{1-n}{2}}me^{-2\cos\theta}}{(1+e^{-4\cos\theta}+2e^{-2\cos\theta}\cos(2\sin\theta))^{\frac{3-n}{4}}}\right)-(\omega(\theta))^\delta\\
    &\geq |\beta|m\chi_n(\theta)-(\omega(\theta))^\delta,
\end{align*}
where $$\chi_n(\theta):=\dfrac{2^{\frac{1-n}{2}}e^{-2\cos\theta}}{(1+e^{-4\cos\theta}+2e^{-2\cos\theta}\cos(2\sin\theta))^{\frac{3-n}{4}}}.$$ After a simple computation we obtain that $\min\chi_n(\theta)=\chi_n(0)$ for all $n\geq1.$
Since $m\geq1$ and  $|\beta| \geq (R_0+\omega_{\theta_{0}}^\delta)(1+e^2)^{(3-n)/2}/(\sqrt{2}e)^{1-n}$, we have \begin{equation*}
    |w|\geq |\beta|\left(\dfrac{(\sqrt{2}e)^{1-n}}{(1+e^2)^{\frac{3-n}{2}}}\right)-\omega_{\theta_{0}}^\delta\geq R_0>\sqrt{2}.
\end{equation*} Therefore, the result follows immediately from Lemma \ref{subradi}.
\end{proof}
\begin{corollary}
    Let $n$ be a positive integer, $\delta\in\{0,1\}$, $R_{0}=e\sqrt{2/(e^2-1)}$, $\omega_{\theta_0}\simeq1.438$ and  $\beta\in\mathbb{C}$ such that $|\beta| \geq (R_0+\omega_{\theta_0}^\delta)(1+e^2)^{(3-n)/2}/(\sqrt{2}e)^{1-n}$. Suppose $f$ be a function in $\mathcal{A}$ such that
    \begin{equation*}
        \left(\dfrac{zf'(z)}{f(z)}\right)^{\delta}+\beta\left(\dfrac{zf'(z)}{f(z)}\right)^{1-n}\left(1+\dfrac{zf''(z)}{f'(z)}-\dfrac{zf'(z)}{f(z)}\right)\prec  \mathfrak{B}(z),
    \end{equation*}then $f\in\mathcal{S}^*_{\mathfrak{B}}.$
\end{corollary}
\begin{proof}
    Let $p(z)=zf'(z)/f(z)$, then 
    $$ (p(z))^\delta+\beta \dfrac{zp'(z)}{(p(z))^n}=\left(\dfrac{zf'(z)}{f(z)}\right)^{\delta}+\beta\left(\dfrac{zf'(z)}{f(z)}\right)^{1-n}\left(1+\dfrac{zf''(z)}{f'(z)}-\dfrac{zf'(z)}{f(z)}\right).$$
    Now the result  follows by an application of Theorem \ref{1stthm2}.
\end{proof}

\begin{theorem}\label{1stthm4}
Let $R_{0}=e\sqrt{2/(e^2-1)}$, $\omega_{\theta_0}\simeq1.438$ and $\beta,\gamma\in\mathbb{C}$   such that $\beta\neq0$ with $|\beta|\omega_{\theta_0} +|\gamma|\leq \sqrt{2}e/((1+e^2)^{3/2}(R_0+\omega_{\theta_0}))$. Suppose $p(z)\in\mathcal{H}$ and satisfies the subordination
\begin{equation*}
    p(z)+ \dfrac{zp'(z)}{\beta p(z)+\gamma}\prec \mathfrak{B}(z),
\end{equation*}then $p(z)\prec \mathfrak{B}(z).$
\end{theorem}

\begin{proof}
   Let $q(z)=\sqrt{2/(1+e^{-2z})}=:\mathfrak{B}(z)$ and $\Omega=q(\mathbb{D})=\Omega_{\mathfrak{B}}.$ Now consider $\psi:\mathbb{C}^2\times\mathbb{D}\to\mathbb{C}$ be as $\psi(r,s;z)=r+ s/(\beta r+\gamma).$ For $\psi$ to lie in the admissible class $\Psi(\Omega_{\mathfrak{B}},\mathfrak{B}),$ we must have $$\psi\left(\sqrt{\dfrac{2}{1+e^{-2e^{i \theta}})}},\dfrac{\sqrt{2}m e^{i \theta}e^{-2e^{i \theta}}}{(1+e^{-2e^{i \theta}})^{3/2}};z\right)\not\in\Omega_{\mathfrak{B}}\qquad(z\in\mathbb{D},0\leq\theta<2\pi,m\geq1),$$ which is equivalent to $$\left|\log\left(\dfrac{(r+ s/(\beta r+\gamma))^2}{2-(r+ s/(\beta r+\gamma))^2}\right)\right|\geq2.$$ 
    We now use the Lemma \ref{subradi} to prove the result, for which,  let us assume $w=r+ s/(\beta r+\gamma)$.
From (\ref{ot}) and (\ref{dt}), we have $|r|=\omega(\theta)\leq\omega_{\theta_0}$ and $|s|/m=d(\theta)\geq\min d(\theta)=d(0)$. Thus
\begin{align*}
  |w|&= \bigg |r+ \dfrac{s}{\beta r+\gamma}\bigg|\\&\geq \dfrac{|s|}{|\beta||r|+|\gamma|}-|r|\\
    &\geq \dfrac{m d(0)}{|\beta|\omega_{\theta_0}+|\gamma|}-\omega_{\theta_0}.
\end{align*}
Since $m\geq1$ and $|\beta|\omega_{\theta_0} +|\gamma|\leq \sqrt{2}e/((1+e^2)^{3/2}(R_0+\omega_{\theta_0}))$, we have
 \begin{equation*}
    |w|\geq  \dfrac{\sqrt{2}e}{(1+e^2)^{3/2}(|\beta|\omega_{\theta_0}+|\gamma|)}-\omega_{\theta_0}\geq R_0>\sqrt{2},
\end{equation*} therefore the result follows at once by Lemma \ref{subradi}.
\end{proof}

\begin{corollary}
    Let  $R_{0}=e\sqrt{2/(e^2-1)}$, $\omega_{\theta_0}\simeq1.438$ and $\beta,\gamma\in\mathbb{C}$   such that $\beta\neq0$ with $\omega_{\theta_0}|\beta| +|\gamma|\leq \sqrt{2}e/((1+e^2)^{3/2}(R_0+\omega_{\theta_0}))$. Suppose $f$ be a function in $\mathcal{A}$ such that
    \begin{equation*}
        \dfrac{zf'(z)}{f(z)}\left(1+\left(\beta\dfrac{zf'(z)}{f(z)}+\gamma\right)^{-1} \left(1+\dfrac{zf''(z)}{f'(z)}-\dfrac{zf'(z)}{f(z)}\right)\right)\prec\mathfrak{B}(z),
    \end{equation*}then $f\in\mathcal{S}^*_{\mathfrak{B}}.$
\end{corollary}
\begin{proof}
    Let $p(z)=zf'(z)/f(z)$, then
    $$  p(z)+ \dfrac{zp'(z)}{\beta p(z)+\gamma}=\dfrac{zf'(z)}{f(z)}\left(1+\left(\beta\dfrac{zf'(z)}{f(z)}+\gamma\right)^{-1} \left(1+\dfrac{zf''(z)}{f'(z)}-\dfrac{zf'(z)}{f(z)}\right)\right).$$
    Therefore the result holds using Theorem \ref{1stthm4}.
\end{proof}
\begin{theorem}\label{janthm1}
Let $n\in\mathbb{N}$, $\delta\in\{0,1\}$, $\beta\in\mathbb{C}\backslash\{0\}$,  $\omega_{\theta_0}\simeq1.438$ and $-1< B\leq0< A\leq1$. Suppose $p\in\mathcal{H}$ such that it satisfies
 \begin{equation*}
(p(z))^\delta+\beta (zp'(z))^n\prec \frac{1+Az}{1+Bz}
\end{equation*}
where
\begin{equation}\label{janousk-ab1}
 |\beta|\geq\left(\dfrac{(1+e^2)^{3/2}}{\sqrt{2}e}\right)^n \left(\dfrac{1+A}{1+B}+\omega_{\theta_0}^\delta\right)
\end{equation}
then $p(z)\prec\mathfrak{B}(z)$.
\end{theorem}

\begin{proof}
	Let $q(z)=\sqrt{2/(1+e^{-2z})}=:\mathfrak{B}(z)$ and $\Omega=\{w\in\mathbb{C}:|(w-1)/(A-Bw)|<1\}$. Let $\psi:\mathbb{C}^2\times\mathbb{D}\to\mathbb{C}$ be defined as $\psi(r,s;z)=r^\delta+\beta s^n$ where $r=\mathfrak{B}(\xi),$ $s=m\xi\mathfrak{B}'(\xi)=m\sqrt{2}\xi e^{-2\xi}/(1+e^{-2\xi})^{3/2}$ for $z\in\mathbb{D}$, $\xi\in\overline{\mathbb{D}}$ and $m\geq1$.
	 For $\psi$ to be an admissible function and for $\psi\in\Psi[\Omega,\mathfrak{B}]$ we need to show $\psi(\mathbb{D})\not\subseteq\Omega$, which  is equivalent to
	\begin{gather*}
	\left|\frac{r^\delta+\beta s^n-1}{A-B(r^\delta+\beta s^n)}\right|\geq1.
	\end{gather*} 
  We now use the Lemma \ref{janlemma} to prove the result, for which,  let us assume $w=r^\delta+ \beta s^n$.
 From (\ref{ot}) and (\ref{dt}), we have  $|r|=\omega(\theta)\leq\omega_{\theta_0}$ and $|s|/m=d(\theta)\geq\min d(\theta)=d(0)$. Thus
	\begin{align*}
  |w|&= |r^\delta+\beta s^n|\\&\geq |\beta||s|^n-|r|^\delta\\
    &\geq |\beta|m^n (d(0))^n-\omega_{\theta_0}^\delta
\end{align*}
Since $m\geq1$ and using (\ref{janousk-ab1}), we have
 \begin{equation*}
    |w|\geq  |\beta|\left(\dfrac{\sqrt{2}e}{(1+e^2)^{3/2}}\right)^n-\omega_{\theta_0}^\delta\geq \dfrac{1+A}{1+B}>1,
\end{equation*}therefore the result follows at once from Lemma \ref{janlemma}.
\end{proof}
\begin{corollary}
    Let $n\in\mathbb{N}$, $\delta\in\{0,1\}$ and $\beta\in\mathbb{C}\backslash\{0\}$ such that $\beta$ satisfies inequality (\ref{janousk-ab1}). Moreover, let $\omega_{\theta_0}\simeq1.438$ and $-1< B\leq0< A\leq1$. Suppose $f\in\mathcal{A}$ such that
    \begin{equation*}
         \left(\dfrac{zf'(z)}{f(z)}\right)^\delta+\beta\left(\dfrac{zf'(z)}{f(z)}\left(1+\dfrac{zf''(z)}{f'(z)}-\dfrac{zf'(z)}{f(z)}\right)\right)^n\prec  \frac{1+Az}{1+Bz},
    \end{equation*}then $f\in\mathcal{S}^*_{\mathfrak{B}}.$
\end{corollary}
\begin{proof}
    Let $p(z)=zf'(z)/f(z)$, then 
    $$(p(z))^\delta+\beta (zp'(z))^n=\left(\dfrac{zf'(z)}{f(z)}\right)^\delta+\beta\left(\dfrac{zf'(z)}{f(z)}\left(1+\dfrac{zf''(z)}{f'(z)}-\dfrac{zf'(z)}{f(z)}\right)\right)^n.$$
    Now the result follows directly from Theorem \ref{janthm1}.
\end{proof}

\begin{theorem}\label{janthm2}
Let $n\in\mathbb{N}$, $\delta\in\{0,1\}$, $\beta\in\mathbb{C}\backslash\{0\}$,  $\omega_{\theta_0}\simeq1.438$ and $-1\leq B<0\leq A\leq1$. Suppose $p\in\mathcal{H}$ such that it satisfies
 \begin{equation*}
(p(z))^\delta+\beta \frac{zp'(z)}{(p(z))^n}\prec \frac{1+Az}{1+Bz},
\end{equation*}
where
\begin{equation}\label{janousk-ab}
 |\beta|\geq\left(\dfrac{(1+e^2)^{\frac{3-n}{2}}}{(\sqrt{2}e)^{1-n}}\right) \left(\dfrac{1+A}{1+B}+\omega_{\theta_0}^\delta\right).
\end{equation}
Then $p(z)\prec\mathfrak{B}(z)$.
\end{theorem}

\begin{proof}
	Let $q(z)=\sqrt{2/(1+e^{-2z})}=:\mathfrak{B}(z)$ and $\Omega=\{w\in\mathbb{C}:|(w-1)/(A-Bw)|<1\}$. Let $\psi:\mathbb{C}^2\times\mathbb{D}\to\mathbb{C}$ be defined as $\psi(r,s;z)=r^\delta+\beta s/r^n$ where $r=\mathfrak{B}(\xi),$ $s=m\xi\mathfrak{B}'(\xi)=m\sqrt{2}\xi e^{-2\xi}/(1+e^{-2\xi})^{3/2}$ for $z\in\mathbb{D}$, $\xi\in\overline{\mathbb{D}}$ and $m\geq1$.
	 For $\psi$ to be an admissible function and for $\psi\in\Psi[\Omega,\mathfrak{B}]$ we need to show $\psi(\mathbb{D})\not\subseteq\Omega$,  which is equivalent to
	\begin{gather*}
	\left|\frac{r^\delta+\beta \frac{s}{r^n}-1}{A-B(r^\delta+\beta \frac{s}{r^n})}\right|\geq1.
	\end{gather*}  
  We now use the Lemma \ref{janlemma} to prove the result, for which,  let us assume $w=r^\delta+ \beta s/r^n$.
	From (\ref{ot}), we have   $|r|=\omega(\theta)\leq\omega_{\theta_0}$. Thus
\begin{align*}
  |w|&= \bigg |r^\delta+\beta \dfrac{s}{r^n}\bigg|\\&\geq |\beta| \left(\dfrac{2^{\frac{1-n}{2}}me^{-2\cos\theta}}{(1+e^{-4\cos\theta}+2e^{-2\cos\theta}\cos(2\sin\theta))^{\frac{3-n}{4}}}\right)-(\omega(\theta))^\delta\\
    &\geq |\beta|m\chi_n(\theta)-(\omega(\theta))^\delta,
\end{align*}
where $$\chi_n(\theta):=\dfrac{2^{\frac{1-n}{2}}e^{-2\cos\theta}}{(1+e^{-4\cos\theta}+2e^{-2\cos\theta}\cos(2\sin\theta))^{\frac{3-n}{4}}}.$$ After a simple computation we obtain that $\min\chi_n(\theta)=\chi_n(0)$ for all $n\geq1.$
Since $m\geq1$ and using (\ref{janousk-ab}), we have
 \begin{equation*}
    |w|\geq  |\beta|\left(\dfrac{(\sqrt{2}e)^{1-n}}{(1+e^2)^{\frac{3-n}{2}}}\right)-\omega_{\theta_0}^\delta\geq \dfrac{1+A}{1+B}>1,
\end{equation*}therefore the result follows from Lemma \ref{janlemma}.
\end{proof}

\begin{corollary}
   Let $n\in\mathbb{N}$, $\delta\in\{0,1\}$ and $\beta\in\mathbb{C}\backslash\{0\}$ such that $\beta$ satisfies inequality (\ref{janousk-ab}). Moreover, let  $\omega_{\theta_0}\simeq1.438$ and $-1< B\leq0<A\leq1$.  Suppose $f$ be a function in $\mathcal{A}$ such that
    \begin{equation*}
         \left(\dfrac{zf'(z)}{f(z)}\right)^{\delta}+\beta\left(\dfrac{zf'(z)}{f(z)}\right)^{1-n}\left(1+\dfrac{zf''(z)}{f'(z)}-\dfrac{zf'(z)}{f(z)}\right)\prec  \dfrac{1+Az}{1+Bz},
    \end{equation*}then $f\in\mathcal{S}^*_{\mathfrak{B}}.$
\end{corollary}
\begin{proof}
    Let $p(z)=zf'(z)/f(z)$, then 
    $$ (p(z))^\delta+\beta \dfrac{zp'(z)}{(p(z))^n}= \left(\dfrac{zf'(z)}{f(z)}\right)^{\delta}+\beta\left(\dfrac{zf'(z)}{f(z)}\right)^{1-n}\left(1+\dfrac{zf''(z)}{f'(z)}-\dfrac{zf'(z)}{f(z)}\right).$$
    Now the result follows by an application of Theorem \ref{janthm2}.
\end{proof}

\begin{theorem}\label{janthm3}
Let $\beta,\gamma\in\mathbb{C}$ such that $\beta\neq0$,  $\omega_{\theta_0}\simeq1.438$ and $-1< B\leq0< A\leq1$. Suppose $p\in\mathcal{H}$ such that it satisfies
 \begin{equation*}
p(z)+\frac{zp'(z)}{\beta p(z)+\gamma}\prec \frac{1+Az}{1+Bz}
\end{equation*} 
provided
\begin{equation}\label{janousk-ab2}
|\beta| \omega_{\theta_0}+|\gamma|\leq\left(\dfrac{(1+e^2)^{3/2}}{\sqrt{2}e} \left(\dfrac{1+A}{1+B}+\omega_{\theta_0}\right)\right)^{-1}
\end{equation}where 
Then $p(z)\prec\mathfrak{B}(z)$.
\end{theorem}

\begin{proof}
	Let $q(z)=\sqrt{2/(1+e^{-2z})}=:\mathfrak{B}(z)$ and $\Omega=\{w\in\mathbb{C}:|(w-1)/(A-Bw)|<1\}$. Let $\psi:\mathbb{C}^2\times\mathbb{D}\to\mathbb{C}$ be defined as $\psi(r,s;z)=r+s/(\beta r+\gamma)$ where $r=\mathfrak{B}(\xi),$ $s=m\xi\mathfrak{B}'(\xi)=m\sqrt{2}\xi e^{-2\xi}/(1+e^{-2\xi})^{3/2}$ for $z\in\mathbb{D}$, $\xi\in\overline{\mathbb{D}}$ and $m\geq1$.
	 For $\psi$ to be an admissible function and for $\psi\in\Psi[\Omega,\mathfrak{B}]$ we need to show $\psi(\mathbb{D})\not\subseteq\Omega$, which is equivalent to
	\begin{gather*}
	\left|\frac{r+\frac{s}{\beta r+\gamma}}{A-B(r+\frac{s}{\beta r+\gamma})}\right|\geq1.
	\end{gather*} 
  We now use the Lemma \ref{janlemma} to prove the result, for which,  let us assume $w=r+ s/(\beta r+\gamma)$.
	From (\ref{ot}) and (\ref{dt}), we have $|r|=\omega(\theta)\leq\omega_{\theta_0}$ and $|s|/m=d(\theta)\geq\min d(\theta)=d(0)$. Thus
\begin{align*}
  |w|&= \bigg |r+ \dfrac{s}{\beta r+\gamma}\bigg|\\&\geq \dfrac{|s|}{|\beta||r|+|\gamma|}-|r|\\
    &\geq \dfrac{m d(0)}{|\beta|\omega_{\theta_0}+|\gamma|}-\omega_{\theta_0}.
\end{align*}
Since $m\geq1$ and using (\ref{janousk-ab2}), we have
 \begin{equation*}
   |w|\geq \dfrac{\sqrt{2}e}{(1+e^2)^{3/2}(|\beta|\omega_{\theta_0}+|\gamma|)}-\omega_{\theta_0}\geq \dfrac{1+A}{1+B}>1,
\end{equation*}therefore the result follows at once from Lemma \ref{janlemma}.
\end{proof}
\begin{corollary}
    Let $\beta,\gamma\in\mathbb{C}$ such that $\beta\neq0$ and satisfies inequality (\ref{janousk-ab2}). Moreover, let  $\omega_{\theta_0}\simeq1.438$ and $-1< B\leq0< A\leq1$. Suppose $f$ be a function in $\mathcal{A}$ such that
    \begin{equation*}
        \dfrac{zf'(z)}{f(z)}\left(1+\left(\beta\dfrac{zf'(z)}{f(z)}+\gamma\right)^{-1} \left(1+\dfrac{zf''(z)}{f'(z)}-\dfrac{zf'(z)}{f(z)}\right)\right)\prec\dfrac{1+Az}{1+Bz},
    \end{equation*}then $f\in\mathcal{S}^*_{\mathfrak{B}}.$
\end{corollary}
\begin{proof}
    Let $p(z)=zf'(z)/f(z)$, then
    $$  p(z)+ \dfrac{zp'(z)}{\beta p(z)+\gamma}=\dfrac{zf'(z)}{f(z)}\left(1+\left(\beta\dfrac{zf'(z)}{f(z)}+\gamma\right)^{-1} \left(1+\dfrac{zf''(z)}{f'(z)}-\dfrac{zf'(z)}{f(z)}\right)\right).$$
    Thus the result follows by an application of Theorem \ref{janthm3}.
\end{proof}
\begin{theorem}\label{lemnthm1}
	Let $n\in\mathbb{N}$, $\delta\in\{0,1\}$, $\beta\in\mathbb{C}\backslash\{0\}$ and  $\omega_{\theta_0}\simeq1.438$. Suppose $p\in\mathcal{H}$ such that it satisfies
\begin{equation*}
(p(z))^\delta+\beta (zp'(z)^n\prec\sqrt{1+z}
\end{equation*}
where
\begin{equation}\label{lemni-1}
 |\beta|\geq\left(\dfrac{(1+e^2)^{3/2}}{\sqrt{2}e}\right)^n \left(\sqrt{2}+\omega_{\theta_0}^\delta\right), 
\end{equation}
then $p(z)\prec\mathfrak{B}(z)$.
\end{theorem}

\begin{proof}
	Let $q(z)=\sqrt{2/(1+e^{-2z})}=:\mathfrak{B}(z)$ and $\Omega=\{w\in\mathbb{C}:|w^2-1|<1\}$. Let $\psi:\mathbb{C}^2\times\mathbb{D}\to\mathbb{C}$ be defined as $\psi(r,s;z)=r^\delta+\beta s^n$ where $r=\mathfrak{B}(\xi),$ $s=m\xi\mathfrak{B}'(\xi)=m\sqrt{2}\xi e^{-2\xi}/(1+e^{-2\xi})^{3/2}$ for $z\in\mathbb{D}$, $\xi\in\overline{\mathbb{D}}$ and $m\geq1$.
	 For $\psi$ to be an admissible function and for $\psi\in\Psi[\Omega,\mathfrak{B}]$ we need to show $\psi(\mathbb{D})\not\subseteq\Omega$, which is equivalent to
	\begin{gather*}
	\left|(r^\delta+\beta s^n)^2-1\right|\geq1.
	\end{gather*} 
  We now use the Lemma \ref{lemlemma} to prove the result, for which,  let us assume $w=r^\delta+ \beta s^n$.
 From (\ref{ot}) and (\ref{dt}), we have $|r|=\omega(\theta)\leq\omega_{\theta_0}$ and $|s|/m=d(\theta)\geq\min d(\theta)=d(0)$. Thus
	\begin{align*}
|w|&=   |r^\delta+\beta s^n|\\&\geq |\beta||s|^n-|r|^\delta\\
    &\geq |\beta|m^n (d(0))^n-\omega_{\theta_0}^\delta.
\end{align*}
Since $m\geq1$ and using (\ref{lemni-1}), we have
 \begin{equation*}
    |w|\geq  |\beta|\left(\dfrac{\sqrt{2}e}{(1+e^2)^{3/2}}\right)^n-\omega_{\theta_0}^\delta\geq \sqrt{2}>1,
\end{equation*}therefore   the result follows from Lemma \ref{lemlemma}.
\end{proof}

\begin{corollary}
	Let $n\in\mathbb{N}$, $\delta\in\{0,1\}$, $\beta\in\mathbb{C}\backslash\{0\}$ such that it satisfies \eqref{lemni-1} and  $\omega_{\theta_0}\simeq1.438$. Further if $f\in\mathcal{A}$ and satisfies
	\begin{equation*}
	 \left(\dfrac{zf'(z)}{f(z)}\right)^\delta+\beta\left(\dfrac{zf'(z)}{f(z)}\left(1+\dfrac{zf''(z)}{f'(z)}-\dfrac{zf'(z)}{f(z)}\right)\right)^n\prec\sqrt{1+z}
	\end{equation*}
	Then $p(z)\in\mathcal{S}^*_{\mathfrak{B}}$.
	\end{corollary}
\begin{proof}
    Let $p(z)=zf'(z)/f(z)$, then 
    $$(p(z))^\delta+\beta (zp'(z))^n=\left(\dfrac{zf'(z)}{f(z)}\right)^\delta+\beta\left(\dfrac{zf'(z)}{f(z)}\left(1+\dfrac{zf''(z)}{f'(z)}-\dfrac{zf'(z)}{f(z)}\right)\right)^n.$$
    Now the result follows from  Theorem \ref{lemnthm1}.
\end{proof}

\begin{theorem}\label{lemnthm2}
Let $n\in\mathbb{N}$, $\delta\in\{0,1\}$, $\beta\in\mathbb{C}\backslash\{0\}$ and  $\omega_{\theta_0}\simeq1.438$. Suppose $p\in\mathcal{H}$ such that it satisfies
 \begin{equation*}
(p(z))^\delta+\beta \frac{zp'(z)}{(p(z))^n}\prec \sqrt{1+z}
\end{equation*}
where
\begin{equation}\label{lemni-2}
 |\beta|\geq\left(\dfrac{(1+e^2)^{\frac{3-n}{2}}}{(\sqrt{2}e)^{1-n}}\right) \left(\sqrt{2}+\omega_{\theta_0}^\delta\right)
\end{equation}
then $p(z)\prec\mathfrak{B}(z)$.
\end{theorem}

\begin{proof}
	Let $q(z)=\sqrt{2/(1+e^{-2z})}=:\mathfrak{B}(z)$ and $\Omega=\{w\in\mathbb{C}:|w^2-1|<1\}$. Let $\psi:\mathbb{C}^2\times\mathbb{D}\to\mathbb{C}$ be defined as $\psi(r,s;z)=r^\delta+\beta s/r^n$ where $r=\mathfrak{B}(\xi),$ $s=m\xi\mathfrak{B}'(\xi)=m\sqrt{2}\xi e^{-2\xi}/(1+e^{-2\xi})^{3/2}$ for $z\in\mathbb{D}$, $\xi\in\overline{\mathbb{D}}$ and $m\geq1$.
	 For $\psi$ to be an admissible function and for $\psi\in\Psi[\Omega,\mathfrak{B}]$ we need to show $\psi(\mathbb{D})\not\subseteq\Omega$, which is equivalent to
	\begin{gather*}
	\left|\left(r^\delta+\beta \frac{s}{r^n}\right)^2-1\right|\geq1.
	\end{gather*}  
 We now use the Lemma \ref{lemlemma} to prove the result, for which,  let us assume $w=r^\delta+ \beta s/r^n$.
	From (\ref{ot}), we have   $|r|=\omega(\theta)\leq\omega_{\theta_0}$. Thus
\begin{align*}
   |w|&=\bigg |r^\delta+\beta \dfrac{s}{r^n}\bigg|\\&\geq |\beta| \left(\dfrac{2^{\frac{1-n}{2}}me^{-2\cos\theta}}{(1+e^{-4\cos\theta}+2e^{-2\cos\theta}\cos(2\sin\theta))^{\frac{3-n}{4}}}\right)-(\omega(\theta))^\delta\\
    &\geq |\beta|m\chi_n(\theta)-(\omega(\theta))^\delta,
\end{align*}
where $$\chi_n(\theta):=\dfrac{2^{\frac{1-n}{2}}e^{-2\cos\theta}}{(1+e^{-4\cos\theta}+2e^{-2\cos\theta}\cos(2\sin\theta))^{\frac{3-n}{4}}}.$$ After a simple computation we obtain that $\min\chi_n(\theta)=\chi_n(0)$ for all $n\geq1.$
Since $m\geq1$ and using (\ref{lemni-2}), we have
 \begin{equation*}
  |w|\geq  |\beta|\left(\dfrac{(\sqrt{2}e)^{1-n}}{(1+e^2)^{\frac{3-n}{2}}}\right)-\omega_{\theta_0}^\delta\geq \sqrt{2}>1,
\end{equation*}therefore the result follows from Lemma \ref{lemlemma}.
\end{proof}

\begin{corollary}
   Let $n\in\mathbb{N}$, $\delta\in\{0,1\}$, $\omega_{\theta_0}\simeq1.438$ and $\beta\in\mathbb{C}\backslash\{0\}$ such that $\beta$ satisfies inequality (\ref{lemni-2}).  Suppose $f$ be a function in $\mathcal{A}$ such that
    \begin{equation*}
         \left(\dfrac{zf'(z)}{f(z)}\right)^{\delta}+\beta\left(\dfrac{zf'(z)}{f(z)}\right)^{1-n}\left(1+\dfrac{zf''(z)}{f'(z)}-\dfrac{zf'(z)}{f(z)}\right)\prec  \sqrt{1+z},
    \end{equation*}then $f\in\mathcal{S}^*_{\mathfrak{B}}.$
\end{corollary}
\begin{proof}
    Let $p(z)=zf'(z)/f(z)$, then 
    $$ (p(z))^\delta+\beta \dfrac{zp'(z)}{(p(z))^n}= \left(\dfrac{zf'(z)}{f(z)}\right)^{\delta}+\beta\left(\dfrac{zf'(z)}{f(z)}\right)^{1-n}\left(1+\dfrac{zf''(z)}{f'(z)}-\dfrac{zf'(z)}{f(z)}\right).$$
Now the result follows as an application of Theorem \ref{lemnthm2}.
\end{proof}

\begin{theorem}\label{lemnthm3}
Let $\beta,\gamma\in\mathbb{C}$ such that $\beta\neq0$ and $\omega_{\theta_0}\simeq1.438$. Suppose $p\in\mathcal{H}$ such that it satisfies
 \begin{equation*}
p(z)+\frac{zp'(z)}{\beta p(z)+\gamma}\prec \sqrt{1+z}
\end{equation*} 
provided
\begin{equation}\label{lemni-3}
|\beta| \omega_{\theta_0}+|\gamma|\leq\left(\dfrac{(1+e^2)^{3/2}}{\sqrt{2}e} \left( \sqrt{2}+\omega_{\theta_0}\right)\right)^{-1}
\end{equation}where 
Then $p(z)\prec\mathfrak{B}(z)$.
\end{theorem}

\begin{proof}
	Let $q(z)=\sqrt{2/(1+e^{-2z})}=:\mathfrak{B}(z)$ and $\Omega=\{w\in\mathbb{C}:|w^2-1|<1\}$. Let $\psi:\mathbb{C}^2\times\mathbb{D}\to\mathbb{C}$ be defined as $\psi(r,s;z)=r+s/(\beta r+\gamma)$ where $r=\mathfrak{B}(\xi),$ $s=m\xi\mathfrak{B}'(\xi)=m\sqrt{2}\xi e^{-2\xi}/(1+e^{-2\xi})^{3/2}$ for $z\in\mathbb{D}$, $\xi\in\overline{\mathbb{D}}$ and $m\geq1$.
	 For $\psi$ to be an admissible function and for $\psi\in\Psi[\Omega,\mathfrak{B}]$ we need to show $\psi(\mathbb{D})\not\subseteq\Omega$, i.e.
	\begin{gather*}
	\left|\left(r+\frac{s}{\beta r+\gamma}\right)^2-1\right|\geq1.
	\end{gather*} 
	We now use the Lemma \ref{lemlemma} to prove the result, for which,  let us assume $w=r+s/(\beta r+\gamma)$. From (\ref{ot}) and (\ref{dt}), we have  $|s|/m=d(\theta)\geq\min d(\theta)=d(0)$ and $|r|=\omega(\theta)\leq\omega_{\theta_0}$. Thus
\begin{align*}
 |w|&=  \bigg |r+ \dfrac{s}{\beta r+\gamma}\bigg|\\&\geq \dfrac{|s|}{|\beta||r|+|\gamma|}-|r|\\
    &\geq \dfrac{m d(0)}{|\beta|\omega_{\theta_0}+|\gamma|}-\omega_{\theta_0}.
\end{align*}
Since $m\geq1$ and using (\ref{lemni-3}), we have
 \begin{equation*}
  |w|\geq \dfrac{\sqrt{2}e}{(1+e^2)^{3/2}(|\beta|\omega_{\theta_0}+|\gamma|)}-\omega_{\theta_0}\geq \sqrt{2}>1,
\end{equation*}therefore,  by Lemma \ref{lemlemma}, the result holds.
\end{proof}
\begin{corollary}
    Let $\omega_{\theta_0}\simeq1.438$ and $\beta,\gamma\in\mathbb{C}$ such that $\beta\neq0$ and satisfies inequality (\ref{lemni-3}).  Suppose $f$ be a function in $\mathcal{A}$ such that
    \begin{equation*}
        \dfrac{zf'(z)}{f(z)}\left(1+\left(\beta\dfrac{zf'(z)}{f(z)}+\gamma\right)^{-1} \left(1+\dfrac{zf''(z)}{f'(z)}-\dfrac{zf'(z)}{f(z)}\right)\right)\prec \sqrt{1+z},
    \end{equation*}then $f\in\mathcal{S}^*_{\mathfrak{B}}.$
\end{corollary}
\begin{proof}
    Let $p(z)=zf'(z)/f(z)$, then
    $$  p(z)+ \dfrac{zp'(z)}{\beta p(z)+\gamma}=\dfrac{zf'(z)}{f(z)}\left(1+\left(\beta\dfrac{zf'(z)}{f(z)}+\gamma\right)^{-1} \left(1+\dfrac{zf''(z)}{f'(z)}-\dfrac{zf'(z)}{f(z)}\right)\right).$$
    Therefore the result holds using Theorem \ref{lemnthm3}.
\end{proof}

\section{Second order differential subordination}

In this section, we consider the following second order differential subordination implication
$$(p(z))^\delta+\gamma zp'(z)+\beta z^2p''(z)\prec h(z)\implies p(z)\prec \mathfrak{B}(z).$$
Taking $h(z)$ as $\mathfrak{B}(z)$, $(1+Az)/(1+Bz)$ $(-1< B\leq0< A\leq1)$ and $\sqrt{1+z}.$

\begin{theorem}
Let $\delta\in\{0,1\}$, $R_{0}=e\sqrt{2/(e^2-1)}$, $\omega_{\theta_0}\simeq1.438$ and $\beta,\gamma\in\mathbb{C}\backslash\{0\}$. If $p(z)\in\mathcal{H}$ such that
$$(p(z))^\delta+\gamma zp'(z)+\beta z^2p''(z)\prec\mathfrak{B}(z),$$ provided 
\begin{equation}\label{bean3}\sqrt{2}e(\gamma(1+e^2)^2+\beta(1-2e^4-e^2))\geq(R_0+\omega_{\theta_0}^\delta)(e^2+1)^{7/2}\end{equation}
then $p(z)\prec \mathfrak{B}(z).$
\end{theorem}
\begin{proof}
Let $q(z)=\mathfrak{B}(z)$ and $\Omega=q(\mathbb{D})=\Omega_{\mathfrak{B}}.$ Now consider $\psi:\mathbb{C}^3\times\mathbb{D}\to\mathbb{C}$ be as $\psi(r,s,t;z)=r^\delta+\gamma s+\beta t.$ For $\psi$ to lie in the admissible class $\Psi(\Omega_{\mathfrak{B}},\mathfrak{B}),$ we must have $$\psi(r,s,t;z)\not\in\Omega_{\mathfrak{B}},$$
 whenever
 $$r=\mathfrak{B}(\zeta)=\sqrt{\dfrac{2}{1+e^{-2e^{i\theta}}}};\;\;\;s=m\zeta\mathfrak{B}'(\zeta)=\dfrac{\sqrt{2}me^{i\theta}e^{-2e^{i\theta}}}{(1+e^{-2e^{i\theta}})^{3/4}};\;\;\;\RE\left(1+\dfrac{s}{t}\right)=m(1+g(\theta)),$$
 where $z\in\mathbb{D}$, $0\leq\theta<2\pi$ and $m\geq1.$ Thus, it is sufficient to show that  $$\left|\log\left(\dfrac{(r^\delta+\gamma s+\beta t)^2}{2-(r^\delta+\gamma s+\beta t)^2}\right)\right|\geq2.$$ 
 We now use the Lemma \ref{subradi} to prove the result, for which,  let us assume $w=r^\delta+\gamma s+\beta t$.
 From (\ref{gt}), (\ref{ot}) and (\ref{dt}), we have $\min g(\theta)=g(0)$, $|s|/m=d(\theta)\geq\min d(\theta)=d(0)$ and $|r|=\omega(\theta)\leq\omega_{\theta_0}$. Thus
\begin{align*}
|w|&=|r^\delta+\gamma s+\beta t|\\
&\geq\gamma |s|\RE\left(1+\dfrac{\beta}{\gamma}\dfrac{t}{s}\right)-|r|^\delta\\
&\geq \gamma m d(\theta)\left(1+\dfrac{\beta}{\gamma}(m g(\theta)+m-1)\right)-(\omega(\theta))^\delta\\
&\geq m\dfrac{\sqrt{2}e}{(e^2+1)^{3/2}}\left(\gamma+\beta\left(m\dfrac{1-2e^4-e^2}{(e^2+1)^2}+m-1\right)\right)-\omega_{\theta_0}^\delta.
\end{align*}Since  $m\geq1$ and using  (\ref{bean3}), we have
\begin{equation*}|w|\geq\dfrac{\sqrt{2}e}{(e^2+1)^{3/2}}\left(\gamma+\beta\left(\dfrac{1-2e^4-e^2}{(e^2+1)^2}\right)\right)-\omega_{\theta_0}^\delta\geq R_0>\sqrt{2},
\end{equation*} therefore the result follows at once by Lemma \ref{subradi}.
\end{proof}

\begin{theorem}
Let $\delta\in\{0,1\}$,  $\omega_{\theta_0}\simeq1.438$ and $\beta,\gamma\in\mathbb{C}\backslash\{0\}$. If $p(z)\in\mathcal{H}$ such that
$$(p(z))^\delta+\gamma zp'(z)+\beta z^2p''(z)\prec\dfrac{1+Az}{1+Bz},\;\;\;-1< B\leq0<A\leq1,$$
provided \begin{equation}\label{jan3}\sqrt{2}e(\gamma(1+e^2)^2+\beta(1-2e^4-e^2))\geq\left(\frac{1+A}{1+B}+\omega_{\theta_0}^\delta\right)(e^2+1)^{7/2}\end{equation}
then $p(z)\prec \mathfrak{B}(z).$
\end{theorem}
\begin{proof}
 Here $q(z)=\mathfrak{B}(z)$ and $h(z)=(1+Az)/(1+Bz)$, thus $$\Omega=h(\mathbb{D})=\left\{w\in\mathbb{C}:\left|\dfrac{w-1}{A-Bw}\right|<1\right\},$$ Now consider $\psi:\mathbb{C}^3\times\mathbb{D}\to\mathbb{C}$ be as $\psi(r,s,t;z)=r^\delta+ \gamma s+\beta  t.$ For $\psi$ to lie in the admissible class $\Psi(\Omega,\mathfrak{B}),$  we must have $$\psi(r,s,t;z)\not\in\Omega,$$
 whenever
 $$r=\mathfrak{B}(\zeta)=\sqrt{\dfrac{2}{1+e^{-2e^{i\theta}}}};\;\;\;s=m\zeta\mathfrak{B}'(\zeta)=\dfrac{\sqrt{2}me^{i\theta}e^{-2e^{i\theta}}}{(1+e^{-2e^{i\theta}})^{3/4}};\;\;\;\RE\left(1+\dfrac{s}{t}\right)=m(1+g(\theta)),$$
 where $z\in\mathbb{D}$, $0\leq\theta<2\pi$ and $m\geq1.$ i.e. it is sufficient to show that $$\left|\dfrac{r^\delta+ \gamma s+\beta  t-1}{A-B(r^\delta+ \gamma s+\beta  t)}\right|\geq1.$$
 We now use the Lemma \ref{janlemma} to prove the result, for which,  let us assume $w=r^\delta+\gamma s+\beta t$. From (\ref{gt}), (\ref{ot}) and (\ref{dt}), we have $\min g(\theta)=g(0)$, $|s|/m=d(\theta)\geq\min d(\theta)=d(0)$ and $|r|=\omega(\theta)\leq\omega_{\theta_0}$. Thus
\begin{align*}
|w|&=|r^\delta+\gamma s+\beta t|\\&\geq\gamma |s|\RE\left(1+\dfrac{\beta}{\gamma}\dfrac{t}{s}\right)-|r|^\delta\\
&\geq \gamma m d(\theta)\left(1+\dfrac{\beta}{\gamma}(m g(\theta)+m-1)\right)-(\omega(\theta))^\delta\\
&\geq m\dfrac{\sqrt{2}e}{(e^2+1)^{3/2}}\left(\gamma+\beta\left(m\dfrac{1-2e^4-e^2}{(e^2+1)^2}+m-1\right)\right)-\omega_{\theta_0}^\delta.
\end{align*}Since  $m\geq1$ and by (\ref{jan3}), we have
$$|w|\geq\dfrac{\sqrt{2}e}{(e^2+1)^{3/2}}\left(\gamma+\beta\left(\dfrac{1-2e^4-e^2}{(e^2+1)^2}\right)\right)-\omega_{\theta_0}^\delta\geq \dfrac{1+A}{1+B}>1,$$
therefore,  by Lemma \ref{janlemma}, the result holds.
\end{proof}

\begin{theorem}
Let $\delta\in\{0,1\}$,  $\omega_{\theta_0}\simeq1.438$ and $\beta,\gamma\in\mathbb{C}\backslash\{0\}$. If $p(z)\in\mathcal{H}$ such that
$$(p(z))^\delta+\gamma zp'(z)+\beta z^2p''(z)\prec\sqrt{1+z},$$
provided \begin{equation}\label{lemn3}\sqrt{2}e(\gamma(1+e^2)^2+\beta(1-2e^4-e^2))\geq\left(\sqrt{2}+\omega_{\theta_0}^\delta\right)(e^2+1)^{7/2}\end{equation}
then $p(z)\prec \mathfrak{B}(z).$
\end{theorem}
\begin{proof}
 Here $q(z)=\mathfrak{B}(z)$ and $h(z)=\sqrt{1+z}$, thus $$\Omega=h(\mathbb{D})=\left\{w\in\mathbb{C}:\left|w^2-1\right|<1\right\},$$ Now consider $\psi:\mathbb{C}^3\times\mathbb{D}\to\mathbb{C}$ be as $\psi(r,s,t;z)=r^\delta+ \gamma s+\beta  t.$ For $\psi$ to lie in the admissible class $\Psi(\Omega,\mathfrak{B}),$  we must have $$\psi(r,s,t;z)\not\in\Omega,$$
 whenever
 $$r=\mathfrak{B}(\zeta)=\sqrt{\dfrac{2}{1+e^{-2e^{i\theta}}}};\;\;\;s=m\zeta\mathfrak{B}'(\zeta)=\dfrac{\sqrt{2}me^{i\theta}e^{-2e^{i\theta}}}{(1+e^{-2e^{i\theta}})^{3/4}};\;\;\;\RE\left(1+\dfrac{s}{t}\right)=m(1+g(\theta)),$$
 where $z\in\mathbb{D}$, $0\leq\theta<2\pi$ and $m\geq1.$ i.e. it is sufficient to show that $$\left|\left(r^\delta+ \gamma s+\beta  t\right)^2-1\right|\geq1.$$
 We now use the Lemma \ref{lemlemma} to prove the result, for which,  let us assume $w=r^\delta+\gamma s+\beta t$. From (\ref{gt}), (\ref{ot}) and (\ref{dt}), we have $\min g(\theta)=g(0)$, $|s|/m=d(\theta)\geq\min d(\theta)=d(0)$ and $|r|=\omega(\theta)\leq\omega_{\theta_0}$. Thus
\begin{align*}
|w|&=|r^\delta+\gamma s+\beta t|\\&\geq\gamma |s|\RE\left(1+\dfrac{\beta}{\gamma}\dfrac{t}{s}\right)-|r|^\delta\\
&\geq \gamma m d(\theta)\left(1+\dfrac{\beta}{\gamma}(m g(\theta)+m-1)\right)-(\omega(\theta))^\delta\\
&\geq m\dfrac{\sqrt{2}e}{(e^2+1)^{3/2}}\left(\gamma+\beta\left(m\dfrac{1-2e^4-e^2}{(e^2+1)^2}+m-1\right)\right)-\omega_{\theta_0}^\delta.
\end{align*}Since  $m\geq1$ and by (\ref{lemn3}), we have
$$|w|\geq\dfrac{\sqrt{2}e}{(e^2+1)^{3/2}}\left(\gamma+\beta\left(\dfrac{1-2e^4-e^2}{(e^2+1)^2}\right)\right)-\omega_{\theta_0}^\delta\geq \sqrt{2}>1,$$
therefore,  by Lemma \ref{lemlemma}, the result holds.
\end{proof}

\section*{Declarations}
\textbf{Ethical Approval } Not Applicable\\
\textbf{Conflicts of interest } The authors declare that they have no competing interests.\\
\textbf{Authors' contributions } All authors contributed equally.\\
\textbf{Funding } Not Applicable\\
\textbf{Availability of data and materials } Not Applicable

\end{document}